\documentclass[12pt,leqno]{amsart}
\usepackage{amsmath, amssymb, graphics, arydshln, epsfig}
\usepackage{pstricks, pst-node, pst-tree,eucal}
\usepackage{url}
\usepackage{tikz}

\newtheorem{thm}{Theorem}

\theoremstyle{definition}
\newtheorem{defn}[thm]{Definition}
\newtheorem{ex}[thm]{Example}

\numberwithin{equation}{section}
\numberwithin{thm}{section}

\newcommand\zero{{\ensuremath{e}}}
\newcommand\one{{\ensuremath{f_{-1,-1}}}}

\newcommand\three{{\ensuremath{f_{1,-1}}}}

\newcommand\seven{{\ensuremath{f_{-1,1}}}}

\newcommand\nine{{\ensuremath{f_{1,1}}}}
\newcommand\Rr{\textup{\sc{r}}}
\newcommand\Ll{\textup{\sc{l}}}
\newcommand\Mm{\textup{\sc{m}}}

\newcommand\calJ{\mathcal{J}}
\renewcommand\l{l}

\newcommand{\ep}{\varepsilon}

\newcommand{\qbinom}[3]{\genfrac{[}{]}{0pt}{}{#1}{#2}_{#3}}

\begin{document}
\pagestyle{headings}

\title{Solving multivariate functional equations}

\author{Michael Chon, Christopher R.\ H.\ Hanusa, and Amy Lee}
\address{Department of Mathematics, Queens College (CUNY), 65-30 Kissena Blvd.,\newline Flushing, NY 11367, U.S.A.}
\email{{\tt mchon89@gmail.com}, {\tt chanusa@qc.cuny.edu}, {\tt alee0143@gmail.com}}

\begin{abstract}
This paper presents a new method to solve functional equations of multivariate generating functions, such as 
$$F(r,s)=e(r,s)+xf(r,s)F(1,1)+xg(r,s)F(qr,1)+xh(r,s)F(qr,qs),$$ giving a formula for $F(r,s)$ in terms of a sum over finite sequences.  We use this method to show how one would calculate the coefficients of the generating function for parallelogram polyominoes, which is impractical using other methods.  We also apply this method to answer a question from fully commutative affine permutations.  
\end{abstract}

\subjclass[2010]{Primary 65Q20, 05A15; Secondary 05A05, 05B50, 05C38, 05E15.
}

\keywords{functional equation, functional recurrence, combinatorial statistic, parallelogram polyomino, staircase polyomino, affine permutation, fully commutative}

\maketitle

\section{Introduction}

Some generating functions are most naturally defined by a functional equation; in this article, we discuss a new method for solving certain functional equations involving multiple variables.  

We set the stage with an elementary example: there are many families of combinatorial objects (e.g., binary trees, Dyck paths, triangulations of convex polygons; see \cite{catadd}) whose generating function satisfies the functional equation
\begin{equation*}
C(x) = 1 + x C(x)^2.
\end{equation*}
The generating function $C(x)=\sum_{n\geq0} C_nx^n$ is a formal power series in one variable $x$ that marks some statistic (often size) on the family of combinatorial objects.  We can interpret the above equation to mean that every object in the family can either be represented as an object of size zero (an empty object) or as being composed of two smaller objects from the family.  The solution to the above equation is 
\begin{equation*}
 C(x) = \frac{1-\sqrt{1-4x}}{2x},
\end{equation*}
and the coefficients of the power series expansion of this generating function are the Catalan numbers.  This gives a conceptual reason for the prevalence of Catalan numbers in combinatorics.  

When we investigate multiple statistics at the same time, the functional equations can become more complicated.   Suppose we are trying to solve for the generating function $F(s,t,x,y,q)$ that satisfies a functional equation of the type
\begin{equation}
F(s)=xe(s)+xf(s)F(1)+xg(s)F(qs).
\label{eq:MBMrecurrence}
\end{equation}
Here we have suppressed the variables $t$, $x$, $y$, and $q$ to simplify the notation; the reader should interpret $F(sq)$ as $F(sq,t,x,y,q)$.  In addition, the functions $e$, $f$, and $g$ may be formal power series in all the variables.  The combinatorial significance of a term like $F(sq)$ is that the statistic marked by $q$ increases at the same time as the statistic marked by $s$ as the objects are being built.  

Functional equations of this type arise naturally when enumerating with statistics combinatorial objects such as polyominoes \cite{brakmbm}, plane trees \cite{barcucci-tree}, lattice paths \cite{barcucci_q}, and pattern-avoiding permutations \cite{barcucci}.  
Bousquet-M\'elou proved in \cite[Lemma 2.3]{mbm-columnconvex} that the solution to Equation~\eqref{eq:MBMrecurrence} is
\begin{equation*}
F(s)=\frac{E(s) + E(1)G(s) - E(s)G(1)}{1-G(1)},
\label{eq:MBMsoln}
\end{equation*}
where
\[E(s) = \sum_{n \geq 0} x^{n+1} g(s) g(sq) \cdots g(sq^{n-1}) e(sq^n)\]  and \[G(s) = \sum_{n \geq 0} x^{n+1} g(s) g(sq) \cdots g(sq^{n-1}) f(sq^n).\]  Bousquet-M\'elou used this lemma, as well as a generalization involving derivatives of $F$ with respect to $s$, to find the generating function for various classes of column-convex polyominoes.  Bousquet-M\'elou and Brak's survey on enumerating polyominoes and polygons \cite{brakmbm} is especially recommended reading.

Theorem~\ref{thm:general} gives the solution to generating functions defined by a functional equation that simultaneously replaces multiple variables, such as in
\begin{equation*}
F(r,s)=e(r,s)+xf(r,s)F(1,1)+xg(r,s)F(qr,1)+xh(r,s)F(qr,qs).
\label{eq:139}
\end{equation*}
Our method applies to functional equations in an arbitrarily large number of formal variables $r_1$ through $r_m$ and with arbitrarily many terms in the functional equation where each $r_i$ can be replaced by $1$ or $q^jr_i$ for $j\geq 0$.  

Our proof is rather elementary---we repeatedly apply the functional equation and take the formal power series limit.  Our main result (Theorem~\ref{thm:general}) is stated in terms of a sum over finite sequences.   More importantly, in principle our method allows for the calculation of the coefficients of the generating function.  This is in contrast to Bousquet-M\'elou's result, which gives a quotient of $q$-Bessel functions, both of which are complicated infinite sums.  Our method does not replace Bousquet-M\'elou's method nor the powerful kernel method \cite{banderier}, which applies to many additional functional equations.  We refer the reader interested in other solution methods to Bousquet-M\'elou and Jehanne's \cite{catalytic}.

In Section~\ref{sec:main}, we prove Theorem~\ref{thm:general} and demonstrate how it applies to some simple functional equations including Equation~\eqref{eq:MBMrecurrence}.  In Section~\ref{sec:mbm}, we apply our method to the functional equation of parallelogram polyominoes and manipulate the solution to show how one would find the coefficients of the corresponding generating function.  In Section~\ref{sec:fc}, we apply a trivariate version of Theorem~\ref{thm:general} to the study of fully commutative affine permutations, which was the original motivation for this study.  An original analysis of fully commutative affine permutations in \cite[Lemma 3.12]{HJ1} involves an unwieldy {\em ad hoc} calculation that we are able to replace by working it into the larger framework of solving multivariate functional equations.

\section{Main Result}\label{sec:main}
We develop some notation in order to state our main result. Let $\mathbf{r}$ denote the set of formal variables $\{r_1,\hdots,r_m\}$.  Our focus will be on $F(\mathbf{r})$, a formal power series in formal variables $x$, $q$, and $\mathbf{r}$, which is most easily described using a functional equation, as follows.  


Throughout this article, we let $\mathbf{j}=(j_1,\hdots,j_m)$ denote an $m$-tuple of integers greater than or equal to $-1$ and let $\calJ$ be a set of such $m$-tuples $\mathbf{j}$.  For each $\mathbf{j}\in\calJ$, let $f_\mathbf{j}(\mathbf{r})$ be a formal power series in the formal variables $x$, $q$, and $\mathbf{r}$.  We also let $e(\mathbf{r})$ be a formal power series in $x$, $q$, and $\mathbf{r}$.  We consider the 
 functional equation
\begin{equation}\label{eq:geneq}
F(\mathbf{r})=e(\mathbf{r})+\sum_{\mathbf{j}\in\calJ} xf_{\mathbf{j}}(\mathbf{r})F\big((q^{j_1}r_1)^{\ep(j_1)},\hdots,(q^{j_m}r_m)^{\ep(j_m)}\big),
\end{equation}
where
\[\ep(j)=\left.\begin{cases} 1 & \text{if $j\geq 0$} \\ 0 & \text{if $j=-1$}\end{cases}\right\}.\]  
This setup allows for arbitrarily many terms in the functional equation, where in each term, each variable $r_i$ can be replaced by $1$ (when $j_i=-1$) or $q^{j_i}r_i$ for $j_i\geq 0$.


We always let $J=(\mathbf{j}_{1},\hdots,\mathbf{j}_{n})$ denote a sequence of $m$-tuples, where each $\mathbf{j}_{k}$ is an integer sequence $\mathbf{j}_{k}=(j_{1,k},\hdots,j_{m,k})$.  For such a sequence $J$, we define two families of integers $a_{i,k}$ and $b_{i,k}$ for integers $i$ and $k$ satisfying $1\leq i\leq m$ and $1\leq k\leq n+1$.  First, define $a_{i,1}=0$ and $b_{i,1}=1$ for all $i$ satisfying $1\leq i\leq m$. 


Then, for all $i$ and $k$ satisfying $1\leq i\leq m$ and $2\leq k\leq n+1$, determine $a_{i,k}$ and $b_{i,k}$ in relation to the sequence $(j_{i,1},j_{i,2},\hdots,j_{i,k-1})$ consisting of the $i$-th entry of the first $k-1$ entries of $J$. If $-1$ occurs in this sequence, define $b_{i,k}=0$ and $a_{i,k}$ to be the sum of the entries of the sequence after the last occurrence of $-1$; otherwise define $b_{i,k}=1$ and set $a_{i,k}$ to be the sum of all entries in this sequence.  



\begin{thm}\label{thm:general} 
Under the framework above, an explicit expression for $F(\mathbf{r})$ is 
\begin{equation}\label{eq:soln}
\begin{aligned}
F(\mathbf{r})&=\sum_{n\geq 0} x^{n}\times
\\
&\sum_{\substack{J=(\mathbf{j}_{1},\hdots,\mathbf{j}_{n})\\{\mathbf{j}_k\in\calJ}}} \!\!\!\!e(q^{a_{1,n+1}}r_1^{b_{1,n+1}},\hdots,q^{a_{m,n+1}}r_m^{b_{m,n+1}})\prod_{k=1}^{n} f_{\mathbf{j}_{k}}(q^{a_{1,k}}r_1^{b_{1,k}},\hdots,q^{a_{m,k}}r_m^{b_{m,k}}),
\end{aligned}
\end{equation}
where the second sum is over all length $n$ sequences $J$ of $m$-tuples from $\calJ$.
\end{thm}

\begin{proof}
Applying Equation~\eqref{eq:geneq} generates $|\calJ|$ new occurrences of $F$ for every initial occurrence of $F$; these new occurrences are each indexed by a $\mathbf{j}\in\calJ$ and weighted by $x$ times the function $f_{\mathbf{j}}$.  The only terms that do not continue to expand in successive applications of Equation~\eqref{eq:geneq} are those with a terminal $e(\mathbf{r})$ function.  As such, the new non-expanding terms generated by applying the functional equation an $(n+1)$-st time consist of one term for every sequence of length $n$ on the alphabet $\calJ$, each of which is multiplied by $x^n$.  Taking the formal power series limit of this iterative procedure, the powers of $x$ become arbitrarily large as to cause no occurrences of $F$ to remain.

The exponents $a_{i,k}$ and $b_{i,k}$ arise as the actions of $$F((q^{j_1}r_1)^{\ep(j_1)},\hdots,
(q^{j_m}r_m)^{\ep(j_m)})$$ are followed.  The initial values $a_{i,1}$ and $b_{i,1}$ are so defined since the first expansion is of $F(r_1,\hdots,r_m)$.  Subsequently, the precise sequence $J=(\mathbf{j}_{1},\hdots,\mathbf{j}_{n})$ determines the exponents of $q$ and the presence of the variable $r_i$.  The variable $r_i$ will not be present in any factor after an occurrence of $j_{i,k}=-1$; whereas, the $k$-th application of the recurrence increases the exponent of $q$ in entry $i$ by positive $j_{i,k}$, and resets the exponent of $q$ in entry $i$ to zero when $j_{i,k}=-1$.
\end{proof}

Independently, a $1$-dimensional analogue of our Equation~\eqref{eq:soln} was found by Mansour and Song in \cite[Lemma 2.1]{MansourSong}.  Their work uses the kernel method to solve systems of functional equations that are of a different type than our Equation~\eqref{eq:geneq}.  Both types of functional equation can be considered generalizations of Equation~\eqref{eq:MBMrecurrence}.

\begin{ex} We start with the simple functional equation 
\[F(r)=r+xF(r)+xF(qr).\]  The coefficients $c_{n,h}$ of the generating function $F(x,q,1)=\sum c_{n,h}x^nq^h$ count the number of ways to produce $h$ heads when flipping $n$ coins.  To apply Theorem~\ref{thm:general}, we note that $\calJ=\{0,1\}$ and that $e(r)=r$, $f_0(r)=1$, and $f_1(r)=1$.  We can now calculate $F(x,q,r)$ as in Equation~\eqref{eq:soln}:
\[\begin{aligned}
F(x,q,r)&=\sum_{n\geq 0} x^{n}\!\!\sum_{\substack{(j_1,\hdots,j_{n})\\j_k\in\{0,1\}}}\!\! q^{a_{n+1}}r^{b_{n+1}}\\
& =\sum_{n\geq 0} x^{n}\!\!\sum_{\substack{(j_1,\hdots,j_{n})\\j_k\in\{0,1\}}}\!\! q^{\sum_{k=1}^n j_k}r\\
&=\sum_{n\geq 0} (1+q)^nrx^{n}.\end{aligned}\] 
As expected, \[F(x,q,1)=\sum_{n\geq 0} (1+q)^nx^{n}.\] 
Alternatively, because $$F(x,q,r)=\frac{r}{1-x}+ \frac{x}{1-x}F(x,q,qr),$$ we can expand directly to find the equivalent expression, \[F(x,q,r)=\frac{r}{1-x-qx}.\]
\end{ex}

\begin{ex}\label{ex:MBMrecurrence} Suppose $F(s)$ satisfies 
\[F(s)=xe(s)+xf_{-1}(s)F(1)+xf_1(s)F(qs),\] as in Equation~\eqref{eq:MBMrecurrence}.  Theorem~\ref{thm:general} implies that 
\[F(s)=\sum_{n\geq 0} x^{n+1}\sum_{\substack{J=(j_1,\hdots,j_{n})\\j_k\in\{-1,1\}}} e(q^{a_{n+1}}s^{b_{n+1}})\prod_{k=1}^{n} f_{j_{k}}(q^{a_k}s^{b_{k}}),\]
where the exponents $a_i$ and $b_i$ depend on the exact sequence $J$ on the alphabet $\{-1,1\}$.  For all positive integers $k$, $a_k$ is the number of consecutive $1$s in $J$ ending with entry $k-1$ and 
\[b_k=\left.\begin{cases}1 & \text{if $j_1=\cdots=j_{k-1}=1$} \\
 0 & \text{otherwise}\end{cases}\right\}.\]
\end{ex}

\begin{ex}
A multivariate functional equation with nice symmetry is
\begin{equation}
\label{eq:mv}
\begin{aligned}
F(r,s)=\zero(r,s)&+x\one(r,s)F(1,1)+x\seven(r,s)F(1,qs)\\&+x\three(r,s)F(qr,1)+x\nine(r,s)F(qr,qs).
\end{aligned}
\end{equation}
When solving Equation~\eqref{eq:mv} in the form of Equation~\eqref{eq:soln}, the exponents $a_{i,k}$ have a simple description.  Given a sequence $J=(\mathbf{j}_{1},\hdots,\mathbf{j}_{n})$ on the alphabet
\[\calJ=\big\{\{-1,-1\},\{-1,1\},\{1,-1\},\{1,1\}\big\},\] define $J_i=(j_{i,1},\hdots,j_{i,n})$ for $i\in\{1,2\}$.  The exponent $a_{i,k}$ is the number of consecutive $1$s in $J_i$ ending with entry $k-1$.
\end{ex}

\section{Application: Parallelogram polyominoes}\label{sec:mbm}

A {\em parallelogram polyomino} (also called a {\em staircase polyomino}) is a horizontally and vertically convex union of $1\times 1$ lattice squares that touches the bottom-left and top-right corners of its bounding rectangle. (An example is given in Figure~\ref{fig:poly}.)  
\begin{figure}[b]
\begin{tikzpicture}[scale=.5]
\draw[step=1cm] (0,1) grid (1,3);
\draw[step=1cm] (1,1) grid (2,4);
\draw[step=1cm] (2,1) grid (3,4);
\draw[step=1cm] (3,3) grid (4,5);
\draw[step=1cm] (4,4) grid (5,5);
\end{tikzpicture}
\caption{A parallelogram polyomino with area $a=11$, left height $l=2$, right height $r=1$, width $w=5$, and height $h=4$.  Its contribution to the generating function $F(q,s,t,x,y)$ would be $q^{11}s^2tx^5y^4$.}
\label{fig:poly}
\end{figure}
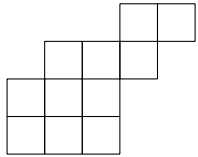
Parallelogram polyominoes are one of the first classes of polyominoes to have been enumerated---P\'olya \cite{PolyaPoly} found the generating function for parallelogram polyominoes by perimeter and area.  One remarkable feature is that the number of parallelogram polyominoes with perimeter $2n$ is the Catalan number $C_{n-1}$.  Indeed, Delest and Viennot \cite{DelestViennot} give a bijection between parallelogram polyominoes with perimeter $2n+2$ and Dyck paths of length $2n$ that at the same time accounts for area.  

For each parallelogram polyomino $P$, we can determine its area $a$, the height of its leftmost column $l$, the height of its rightmost column $r$, and the width $w$ and height $h$ of its bounding rectangle.  Bousquet-M\'elou \cite{mbm-columnconvex} used a layered approach (constructing the polyominoes column by column) to prove that the parallelogram polyomino generating function 
\[F(q,s,t,x,y)=\sum_P q^{a(P)}s^{l(P)}t^{r(P)}x^{w(P)}y^{h(P)}\]  satisfies the functional equation
\[F(s)=\frac{xstyq}{1-styq}+\frac{xsq}{(1-sq)(1-syq)}F(1)+\frac{-xsq}{(1-sq)(1-syq)}F(sq).\]
This is of the form in Equation~\eqref{eq:MBMrecurrence}, discussed in Example~\ref{ex:MBMrecurrence}.  As such,
\[F(s)=\sum_{n \geq 0}x^{n+1}\!\!\!\!\!\sum_{J\in\{-1,1\}^n} \!\!\!\!(-1)^{o(J)}\frac{q^{a_{n+1}+1}s^{b_{n+1}}ty}{1-q^{a_{n+1}+1}s^{b_{n+1}}ty}\prod_{k=1}^{n}\frac{q^{a_{k}+1}s^{b_{k}}}{(1-q^{a_{k}+1}s^{b_k})(1-q^{a_{k}+1}s^{b_k}y)},\]
where the sum is sequences $J$ of length $n$ over the alphabet $\{-1,1\}$ and $o(J)$ is the number of entries of $J$ equal to $1$. Expanding the middle factor and simplifying, 
\begin{equation}
\begin{aligned}
F(s)&=\sum_{n \geq 0}x^{n+1} \!\!\!\!\!\sum_{J\in\{-1,1\}^n}\!\!\!\! (-1)^{o(J)} \sum_{m \ge 1}(tyq^{a_{n+1}+1}s^{b_{n+1}})^m\prod_{k=1}^{n}  \frac{q^{a_{k}+1}s^{b_{k}}}{(1-q^{a_{k}+1}s^{b_k})(1-q^{a_{k}+1}s^{b_k}y)} \\
&=\sum_{n \geq 0} \sum_{m \ge 1} x^{n+1}t^m \!\!\!\!\!\sum_{J\in\{-1,1\}^n} \!\!\!\!(-1)^{o(J)}(yq^{a_{n+1}+1}s^{b_{n+1}})^m\prod_{k=1}^{n}  \frac{q^{a_{k}+1}s^{b_{k}}}{(1-q^{a_{k}+1}s^{b_k})(1-q^{a_{k}+1}s^{b_k}y)}.\\
\end{aligned}
\end{equation}
We conclude that the coefficient of $x^{n+1}t^m$ is
\[\sum_{J\in\{-1,1\}^n}\!\!\!\!(-1)^{o(J)}(yq^{a_{n+1}+1}s^{b_{n+1}})^m\prod_{k=1}^{n}  \frac{q^{a_{k}+1}s^{b_{k}}}{(1-q^{a_k+1}s^{b_k})(1-q^{a_k+1}s^{b_k}y)}.\]
Extraction with respect to $s$ or $y$ is possible by conditioning on the length of the initial run of $1$s in $J$.

\section{Application: Affine permutations}\label{sec:fc}

The methods developed in Section~\ref{sec:main} provide an alternative to a calculation of Hanusa and Jones \cite{HJ1}, who enumerated fully commutative affine permutations by number of inversions.  We do not require the theory of affine permutations; we prove an equivalent result on $(L)(M)(R)$-words. We recall that a permutation $w$ is $321$-avoiding if there do not exist entries $i<j<k$ such that $w_i>w_j>w_k$.

\begin{defn}
Let $L$, $M$, and $R$ be positive integers.  We define an $(L)(M)(R)$-word to be a $321$-avoiding permutation of the set $\{1,2,\hdots,L+M+R\}$ where the numbers $\{1,\hdots,L\}$ are in increasing order, the numbers $\{L+M+1,\hdots,L+M+R\}$ are in increasing order, and every number in the set $\{1,\hdots,L\}$ is to the left of every number in the set $\{L+M+1,\hdots,L+M+R\}$.  
\end{defn}

When the exact value of an integer in a word is of no consequence, we may replace integers in $\{1,\hdots,L\}$ with a letter $\Ll$, integers in $\{L+1,\hdots,L+M\}$ with a letter $\Mm$, and integers in $\{L+M+1,\hdots,L+M+R\}$ with a letter $\Rr$.  

Let $w=w_1w_2\hdots w_n$ be an $(L)(M)(R)$-word.  We say that $(i,j)$ is an {\em inversion} if $i<j$ and $w_i>w_j$ and we say that $w$ has an {\em $\Mm$-descent} in position $i$ if the $i$-th $\Mm$ is larger than the $(i+1)$-st $\Mm$.

\begin{defn}
For an $(L)(M)(R)$-word $w$ with at least one $\Mm$-descent, we define the following combinatorial statistics.  The statistic $n(w)=L+M+R$ is the size of the word, and will be marked by the variable $x$.  (It is also helpful to define $m(w)=M$.)  We let $\l(w)$ be the number of inversions of $w$, marked by the variable $q$.  Then we define $\alpha(w)$ to be the number of $\Mm$'s after the last $\Mm$-descent and before the leftmost $\Rr$ (marked by variable $r$), $\beta(w)$ to be the number of $\Rr$'s before the last $\Mm$ (marked by variable $s$), and $\gamma(w)$ to be the number of $\Mm$'s to the right of the leftmost $\Rr$ (marked by variable $t$).  
\end{defn}

\begin{ex}\label{ex:LMR}
The $(4)(7)(2)$-word 
\[w=1, 2, \mathbf{7}, 3, \mathbf{8}, 4, \mathbf{5}, \mathbf{10}, \mathbf{6}, \mathbf{12}, \mathbf{9}, 13, 14, \mathbf{11}\]
may be written $\Ll\Ll\Mm\Ll\Mm\Ll\Mm\Mm\Mm\Mm\Mm\Rr\Rr\Mm$. (The 7 $\Mm$'s are the bolded numbers in the sequence.) Note that $w$ is $321$-avoiding, has a total of $l(w)=11$ inversions and has $\Mm$-descents in positions $2$ \textup{($8\!\searrow\!5$)}, $4$ \textup{($10\!\searrow\!6$)}, and $6$ \textup{($12\!\searrow\!9$)}.   Last, we see that $\alpha(w)=1$, $\beta(w)=2$, and $\gamma(w)=1$. 
\end{ex}

  
Hanusa and Jones \cite[Lemma 3.12]{HJ1} determine the generating function $\sum x^{n(w)} q^{\l(w)}$ for all $(L)(M)(R)$-words $w$ with at least two $\Mm$-descents. Their formula requires calculating coefficients of a generating function that is a quotient of $q$-Bessel functions and subsequently inserting those coefficients into another generating function.  

As an application of Theorem~\ref{thm:general}, we show that it is possible to calculate an expression for $\sum x^{n(w)} q^{\l(w)}$ that does not require the original awkward substitution.  In the discussion that follows, we suppose that $L$ and $R$ are positive integers and define \[G(x,q,r,s,t)=\sum_w x^{m(w)} q^{\l(w)} r^{\alpha(w)} s^{\beta(w)} t^{\gamma(w)}\] to
 be the generating function for $(L)(M)(R)$-words with at least two $\Mm$-descents.  A natural partition of $(L)(M)(R)$-words with at least two $\Mm$-descents are into those with zero $\Mm$'s to the right of the leftmost $\Rr$, and those with one or more $\Mm$'s to the right of the leftmost $\Rr$.  The former will be counted by $F_0(x,q,r)$ in Theorem~\ref{thm:F0} and the latter will be counted by $F_1(x,q,r,s,t)$ in Theorem~\ref{thm:F1}, so that 
\[G(x,q,r,s,t)=F_0(x,q,r)+F_1(x,q,r,s,t).\]

We recall that the $q$-binomial coefficient $\qbinom{n}{k}{q}$ is defined as $$\qbinom{n}{k}{q}=\frac{(1-q^n)(1-q^{n-1})\cdots(1-q^{n-k+1})}{(1-q^k)(1-q^{k-1})\cdots(1-q)}$$ and enumerates the ways to intersperse $k$ smaller numbers with $n-k$ larger numbers, keeping track of the number of inversions incurred using the variable $q$.

\begin{thm}\label{thm:F0} Suppose $L$ and $R$ are positive integers.  Let $F_0$ be the generating function enumerating $(L)(M)(R)$-words with at least two $\Mm$-descents and with zero $\Mm$'s to the right of the leftmost $\Rr$.  Then $F_0(x,q,r)$ satisfies the functional equation
\begin{equation}\label{eq:affineFC2}
F_0(r)=\frac{E_0(r)}{1-xr}+\frac{xqr}{(1-qr)(1-xr)}F_0(1)-\frac{xqr}{(1-qr)(1-xr)}F_0(qr)
\end{equation}
where
\[
E_0(r)=\sum_{m\geq 0}\sum_{i=1}^{m-1}x^{m+1}\qbinom{L+i}{L}{q}\bigg(\qbinom{m}{i}{q}-1\bigg)\Bigg(\sum_{k=1}^{m-i-1} q^kr^k\Bigg).
\]
\end{thm}

\begin{proof}
Similar to the proof in \cite{HJ1}, we use West's generating tree method \cite{West} to enumerate $(L)(M)(R)$-words with at least two $\Mm$-descents.  We fix $L$ and $R$ and determine all the ways in which we may insert an $(M+1)$-st $\Mm$ into an $(L)(M)(R)$-word (incrementing each number larger than $L+M$ by one), keeping track of how the combinatorial statistics change in each case.  

When we restrict to enumerating $(L)(M)(R)$-words with at least two $\Mm$-descents in which there is no $\Mm$ to the right of an $\Rr$, there is one contiguous set of positions into which we may insert $L+M+1$ that does not create a $321$-pattern---between the rightmost $\Mm$-descent and the leftmost $\Rr$.  We might visualize this as indicated below, where $^{\scriptscriptstyle\searrow}$ represents a descent and $_\bullet$ represents a possible insertion position.

\vspace{-.15in}
\[\Mm\Ll\Ll\Mm\Ll\Mm^{\scriptscriptstyle\searrow}\Mm\Mm^{\scriptscriptstyle\searrow}\Mm\Mm^{\scriptscriptstyle\searrow}{}_{\bullet}\Mm_\bullet\Mm_\bullet\Mm_\bullet\Rr\Rr\Rr\Rr\]

Condition on $k$, the number of $\Mm$'s to the right of where we insert $L+M+1$.  When $k=0$, $m(w)$ and $\alpha(w)$ increase by one and $\l(w)$ stays the same.  As $k$ ranges from $1$ to $\alpha(w)$, then $m(w)$ increases by one, $\l(w)$ increases by $k$, and $\alpha(w)$ becomes $k$.  From this we see that the $(L)(M+1)(R)$-words of this type generated from $(L)(M)(R)$-words of this type have generating function
\begin{equation*}
\begin{aligned}
&\sum_{w}x^{m(w)+1}q^{l(w)}r^{\alpha(w)+1}+\sum_{w}\sum_{k=1}^{\alpha(w)} x^{m(w)+1}q^{l(w)+k}r^k\\\
&=xr\sum_{w}x^{m(w)}q^{l(w)}r^{\alpha(w)}+x\sum_wx^{m(w)}q^{l(w)}\sum_{k=1}^{\alpha(w)}(qr)^k\\
&=xr\sum_{w}x^{m(w)}q^{l(w)}r^{\alpha(w)}+\frac{xqr}{1-qr}\sum_wx^{m(w)}q^{l(w)}\big(1-(qr)^{\alpha(w)}\big)\\
&=xrF_0(r)+\frac{xqr}{1-qr}F_0(1)-\frac{xqr}{1-qr}F_0(qr).\\
\end{aligned}
\end{equation*}
We also need to enumerate the $(L)(M+1)(R)$-words of this type that do not arise in this way.  More specifically, we enumerate $(L)(M+1)(R)$-words $w$ such that removing the number $L+M+1$ (and decrementing each larger element by one) leaves an $(L)(M)(R)$-word $w'$ with one $\Mm$-descent and no $\Mm$ to the right of an $\Rr$.   Suppose that the one $\Mm$-descent in $w'$  occurs in position $i$.  There are $\big(\qbinom{M}{i}{q}-1\big)$ ways to arrange the $\Mm$-entries in this way and $\qbinom{L+i}{L}{q}$ ways to intertwine the $\Ll$'s with the first $i$ $\Mm$'s.  We then reinsert $L+M+1$ to the right the existing $\Mm$-descent.  If there are $k$ $\Mm$'s to the left of $L+M+1$ (where $k$ ranges from $1$ to $m-i-1$), then $\l(w)=\l(w')+k$ and $\alpha(w)=k$.  Summing over all $m$ gives $E_0(r)$.  

Combining these new $(L)(M+1)(R)$-words with the $(L)(M+1)(R)$-words created using the generating tree and subsequently solving for $F_0(r)$ gives Equation~\eqref{eq:affineFC2}.
\end{proof}

\begin{thm}\label{thm:F1}
Suppose $L$ and $R$ are positive integers.  Let $F_1$ be the generating function enumerating $(L)(M)(R)$-words with at least two $\Mm$-descents and with at least one $\Mm$ to the right of the leftmost $\Rr$.  Then $F_1(x,q,r,s,t)$ satisfies the functional equation
\begin{equation}\label{eq:affineFC}
\begin{aligned}
F_1(r,s,t)=E_1(r,s,t)&+\frac{x}{1-qr}F_1(1,s,qt)-\frac{xqr}{1-qr}F_1(qr,s,qt)\\
&-\frac{xt(qs)^{R+1}}{1-qs}F_1(r,1,t)+\frac{xt}{1-qs}F_1(r,qs,t),
\end{aligned}
\end{equation}
where 
\[
\begin{aligned}
E_1(r,s,t)=&\, xt\bigg(\sum_{b=1}^R (qs)^b\bigg)F_0(r)+\sum_{m\geq 0}\sum_{i=1}^{m-1}x^{m+1}\qbinom{L+i}{L}{q}\bigg(\qbinom{m}{i}{q}-1\bigg)\times\\
&\hspace{1.7in}\Bigg(\sum_{k=1}^{m-i-1}\sum_{c=1}^{k}\sum_{b=1}^Rq^{k+b+c-1}r^{k-c}s^bt^{c}\qbinom{b+c-2}{b-1}{q}\Bigg).
\end{aligned}
\]
\end{thm}

\begin{proof}
When there is an $\Mm$ to the right of an $\Rr$, there are two disjoint contiguous sets of positions into which we may insert $L+M+1$:


\vspace{-.15in}
\[\Ll\Mm\Mm\Ll\Ll\Mm\Mm^{\scriptscriptstyle\searrow}\Mm\Mm\Mm^{\scriptscriptstyle\searrow}{}_{\bullet}\Mm_\bullet\Mm_\bullet\Rr\Rr\Mm\Mm\Rr\Mm_\bullet\Rr_\bullet\Rr_\bullet\]

First, we may insert $L+M+1$ into a position between the rightmost $\Mm$-descent and the leftmost $\Rr$.  Suppose that there are $k$ $\Mm$'s to the right of where we insert $L+M+1$ and to the left of the leftmost $\Rr$, where $k$ ranges from $0$ to $\alpha(w)$.  In this case, $m(w)$ increases by one, $\l(w)$ increases by $k+\gamma(w)$, and $\alpha(w)$ becomes $k$, while $\beta(w)$ and $\gamma(w)$ remain unchanged.  This provides a contribution of 
\begin{equation*}
\begin{aligned}
&\sum_{w}\sum_{k=0}^{\alpha(w)} x^{m(w)+1}q^{l(w)+k+\gamma(w)}r^ks^{\beta(w)}t^{\gamma(w)}\\
&=x\sum_wx^{m(w)}q^{l(w)}s^{\beta(w)}(qt)^{\gamma(w)}\sum_{k=0}^{\alpha(w)}(qr)^k\\
&=\frac{x}{1-qr}\sum_wx^{m(w)}q^{l(w)}s^{\beta(w)}(qt)^{\gamma(w)}\big(1-(qr)^{\alpha(w)+1}\big)\\
&=\frac{x}{1-qr}F(x,q,1,s,qt)-\frac{xqr}{1-qr}F(x,q,qr,s,qt).\\
\end{aligned}
\end{equation*}

Alternatively, we may insert $L+M+1$ into a position after the rightmost $\Mm$.  Suppose that there are $k$ $\Rr$'s to the left of where we insert $L+M+1$, where $k$ ranges from $\beta(w)$ to $R$.  In this case, $m(w)$ and $\gamma(w)$ increase by one, $\l(w)$ increases by $k$, $\alpha(w)$ stays the same, and $\beta(w)$ becomes $k$.  This provides a contribution of 
\begin{equation*}
\begin{aligned}
&\sum_{w}\sum_{k=\beta(w)}^{R} x^{m(w)+1}q^{l(w)+k}r^{\alpha(w)}s^{k}t^{\gamma(w)+1}\\
&=xt\sum_wx^{m(w)}q^{l(w)}r^{\alpha(w)}t^{\gamma(w)}\sum_{k=\beta(w)}^{R}(qs)^k\\
&=\frac{xt}{1-qs}\sum_wx^{m(w)}q^{l(w)}r^{\alpha(w)}t^{\gamma(w)}\big((qs)^{\beta(w)}-(qs)^{R+1}\big)\\
&=\frac{xt}{1-qs}F(x,q,r,qs,t)-\frac{xt(qs)^{R+1}}{1-qs}F(x,q,r,1,t).\\
\end{aligned}
\end{equation*}

\medskip
We also must enumerate $(L)(M+1)(R)$-words $w$ with at least two $\Mm$-descents and at least one $\Mm$ to the right of an $\Rr$ such that removing element $L+M+1$ (and decrementing each larger element by one) gives an $(L)(M)(R)$-word $w'$ that is not in this class.  This can occur in two ways---either $L+M+1$ is the only $\Mm$ to the right of the leftmost $\Rr$ in $w$ or $L+M+1$ is to the left of the leftmost $\Rr$ and $w'$ has only one $\Mm$-descent.  

In the former case we insert $L+M+1$ into any position after the leftmost $R$ in an $(L)(M)(R)$-word with at least two $\Mm$-descents and no $\Mm$'s to the right of an $\Rr$.  Conditioning on $b$, the number of $\Rr$'s before $L+M+1$, gives a contribution of $xt\big(\sum_{b=1}^R (qs)^b\big)F_0(r)$.

In the latter case we follow a similar argument to that given for $E_0(r)$, except that after placing $L+M+1$ to the left of $k$ $\Mm$'s (for a contribution of $q^k$),  we intertwine any positive number of $\Rr$'s with a number of $\Mm$'s between $1$ and $k$.  If we intertwine $b$ $\Rr$'s with $c$ $\Mm$'s, we first pass one $\Rr$ past $c$ $\Mm$'s and one $\Mm$ past $b-1$ $\Rr$'s for a contribution of $q^{b+c-1}$, and there are $\qbinom{b+c-2}{b-1}{q}$ ways to intertwine the remaining $\Mm$'s and $\Rr$'s.  With $c$ $\Mm$'s to the right of the leftmost $\Rr$, there are $k-c$ $\Mm$'s after the last $\Mm$-descent and before the leftmost $\Rr$; this accounts for the $r^{k-c}s^bt^c$. \end{proof}

Finally, we note that Equation~\eqref{eq:affineFC2} has the form of Equation~\eqref{eq:MBMrecurrence}, which is solved in Example~\ref{ex:MBMrecurrence}, and that Equation~\eqref{eq:affineFC} is a functional equation of the form
\[\begin{aligned}
F(r,s,t)=e(r,s,t)&+f_{-1,0,1}F(1,s,qt) +f_{1,0,1}F(qr,s,qt)\\&+f_{0,-1,0}F(r,1,t)\hspace{.1in} +f_{0,1,0}F(r,qs,t),
\end{aligned}\]
to which Theorem~\ref{thm:general} also applies directly.  The coefficients of the generating function \[x^{R+L}G(x,q,1,1,1)=\sum x^{n(w)} q^{\l(w)}\] calculated using these functional equations agree with the formula given in \cite{HJ1}.

\section*{Acknowledgments}
We thank Toufik Mansour for discussions about the relationship between our work and his work on the kernel method.  We thank the editor Igor Pak for suggesting improvements in the writing style that greatly improved the clarity of the exposition.
We are grateful for the support of a 2011--2012 Queens College Undergraduate Research and Mentoring Education grant. C.\ R.\ H.\ Hanusa gratefully acknowledges the support of PSC-CUNY grant TRADA-43-127.


\bibliographystyle{siam}
\bibliography{Multivariate}

\end{document}